\documentclass[12pt]{amsart}
\usepackage[all]{xy}
\usepackage{pgf}
\usepackage{tikz}
\usepackage{tikz-cd}
\usepackage[margin=1in]{geometry}  
\usepackage{graphicx}              
\usepackage{amsmath}               
\usepackage{amsfonts}              
\usepackage{amsthm}                
\usepackage{euscript}

\theoremstyle{definition}
\newtheorem{thm}{Theorem}[section]
\newtheorem{lem}[thm]{Lemma}
\newtheorem{prop}[thm]{Proposition}

\newtheorem{rem}[thm]{Remark}
\newtheorem{defn}{Definition}[section]
\newtheorem{exa}[thm]{Example}

\DeclareMathOperator{\op}{op}
\DeclareMathOperator{\Gr}{Gr}
\DeclareMathOperator{\ch}{ch}
\DeclareMathOperator{\id}{id}

\DeclareMathOperator{\Hom}{Hom}
\DeclareMathOperator{\Spec}{Spec}

\DeclareMathOperator{\cone}{Cone}
\DeclareMathOperator{\Star}{Star}
\DeclareMathOperator{\rk}{rk}
\DeclareMathOperator{\divi}{div}
\newcommand{\RR}{\mathbb{R}}      
\newcommand{\ZZ}{\mathbb{Z}}      
\newcommand{\VV}{\mathbb{V}}      
\newcommand{\QQ}{\mathbb{Q}}

\newcommand{\Aff}{\mathbb{A}}
\newcommand{\GG}{\mathbb{G}}
\newcommand{\tbf}{\textbf}
\newcommand{\ol}{\overline}
\newcommand{\A}{\mathbb{A}}
\newcommand{\PP}{\mathbb{P}}
\newcommand{\mc}{\mathcal}
\newcommand{\X}{{\mathcal X}}

\tikzset{node distance=2cm, auto}

\title{The Chow cohomology of affine toric varieties}

\author{Dan Edidin}
\author{Ryan Richey}
\address{Department of Mathematics, University of Missouri, Columbia MO 65211}
\email{edidind@missouri.edu}
\email{}

\thanks{The first author was supported by Simons Collaboration Grant
  315460.}


\begin{document}

\begin{abstract}
  We study the Fulton-Macpherson Chow cohomology of affine toric
  varieties. In particular, we prove that the Chow cohomology vanishes
  in positive degree. We prove an analogous result for the operational
  $K$-theory defined by Anderson and Payne.
\end{abstract}

\maketitle

\section{Introduction}
	
Toric varieties are a rich source of examples for studying delicate
questions in both ordinary and equivariant intersection theory.  Both
the Chow and cohomology rings of smooth, projective toric varieties
are isomorphic to the Stanley-Reisner ring of the associated fan,
while Fulton and Sturmfels \cite{FuSt:97} proved that the Chow
cohomology ring of a complete toric variety is the ring of Minkowski
weights.  Later, Payne \cite{Pay:06} calculated the $T$-equivariant
Chow cohomology ring of an arbitrary toric variety and identified it
with the ring of piecewise polynomial functions on the corresponding
fan.  However, as noted by Katz and Payne \cite{KaPa:08}, the
relationship between equivariant and ordinary Chow cohomology is quite
subtle.

The main result of this paper, Theorem \ref{thm.Chowcoh}, is to show
that the (non-equivariant) Chow cohomology ring of any affine toric
variety vanishes in degree greater than zero.  By contrast, the Chow
homology groups of an affine toric variety need not vanish and can
have both torsion and non-torsion elements.

Although their coordinate rings have prime cycles which are not
rationally equivalent to 0, our result shows that the Chow
cohomology rings of affine toric varieties  behave like that of affine spaces.  As observed by
Gubeladze \cite{Gub:99} singular affine toric varieties constitute a
class of intuitively contractible varieties naturally generalizing
affine spaces.  In particular, all vector bundles are trivial so the
Grothendieck ring is isomorphic to $\ZZ$. In a parallel result we also
prove, Theorem \ref{thm.opKT}, that the operational $K$-theory ring
defined by Anderson and Payne in \cite{AnPa:15} is also isomorphic to
$\ZZ$. Again, the Grothendieck group of coherent sheaves on an affine
toric variety need not be isomorphic to $\ZZ$; see Example \ref{ex.final}.

Our Chow cohomology result fits
into a framework developed in \cite{EdSa:17}. There,
the first author and Satriano show that for spaces, such as toric
varieties, which are good moduli spaces of smooth Artin stacks,
elements of the Chow cohomology ring have canonical representatives as
{\em topologically strong} cycles on the corresponding stack.

Based on the results of \cite{EdSa:17} the authors of the present
paper have a conjectural description of the Chow cohomology ring of an
arbitrary toric variety in terms of strong complete intersection
cycles on the corresponding canonical stack. This will be presented elsewhere
\cite{Ric:19}, but we remark that our results here imply the
conjecture for affine toric varieties.

\subsection{Outline of the proof}
The proof of Theorem \ref{thm.Chowcoh} uses
Kimura's fundamental exact
sequence for Chow cohomology \cite{Kim:92} and
the stack-theoretic
results proved in \cite{EdSa:17}.

If $X(\sigma)$ is the affine toric variety defined by a cone $\sigma$, then
we denote by $\sigma^*$ the fan obtained by star subdivision with
respect to the ray $v = \Star(\sigma)$. The map of toric varieties
$X(\sigma^*) \to X(\sigma)$ is a blowup with exceptional divisor $E=V(\rho_v)$.
Using Kimura's exact sequence for Chow cohomology, we show
that in order to prove that $A^k_{\op}(X(\sigma))=0$, it suffices to prove the
restriction map $A_{\op}^k(X(\sigma^*)) \to A_{\op}^k(E)$ is injective.

Denote by ${\mathcal E}$ and $\X(\sigma^*)$ the corresponding toric canonical 
stacks. Using the results of \cite{EdSa:17} we can show
that $A^k(E)_\QQ \to A^k(X(\sigma^*))_\QQ$ is injective by showing
that the corresponding map of Chow rings of smooth stacks $A^k(\X(\sigma^*))
\to A^k({\mathcal E})$ is injective. By computing the Chow rings
of toric canonical stacks we show directly that the corresponding map
is in fact an isomorphism. This shows that $A^k(X(\sigma))_\QQ = 0$
for $k > 0$. The proof concludes by using results of Hausel and Sturmfels
to show that Chow cohomology of any affine toric variety is torsion free.

The proof for operational $K$-theory uses a similar stategy.  Here
we combine Anderson and Payne's \cite{AnPa:15} results with Theorem \ref{thm.kpullbackinj}
to obtain the desired result.

\subsection{Acknowledgments} The authors are grateful to Sam Payne for suggesting that the corresponding result for operational $K$-theory should hold.

\section{Background on Toric Varieties and canonical stacks}

\subsection{Basic toric terminology}	
Following \cite{Ful:93} or \cite{CLS:11}, let $k$ be be an
algebraically closed field of characteristic zero, $N$ and $M$ be rank
$n$ dual
lattices over $\ZZ$ (which will be the lattices of 1-parameter
subgroups and characters of a torus), and let $\sigma$ be a strongly
convex, rational, and polyhedral cone in $N_{\RR} = N\otimes_{\ZZ}
\RR$.  From $\sigma$, we can construct the dual cone $\sigma \check{}$
in $M_{\RR}$ which satisfies the property that $\sigma \check{}\cap
M$ is a semigroup.  The $\tbf{affine toric variety}$
$X(\sigma)$ is defined as:
$\displaystyle X(\sigma) :=  \Spec(k[(\sigma \check{}\cap M)])$.
		
A fan $\Delta$ in $N_{\RR}$ is a collection of cones satisfying the
following two properties: each face of a cone is in $\Delta$, and the
intersection of every two cones is a face of each.  Each cone in
$\Delta$ generates an affine toric variety which glue together to define a normal toric variety
$X(\Delta)$ with an action of the torus $T_N = X(\{0\})$.
Conversely, it is well known that every normal variety containing a
dense torus whose action extends to an action on the total variety
arises from a fan.

The orbit-cone correspondence associates to each cone $\sigma$
in a fan a $\Delta$ a torus orbit $O(\sigma)$ of dimension
$n - \dim \sigma$. We denote the closure of $O(\sigma)$ by $V(\sigma)$.
The orbit closure is a finite union of orbits corresponding to the faces
of $\sigma$. 

Torus equivariant morphisms of toric varieties $X(\Delta') \rightarrow X(\Delta)$ correspond to morphisms of
lattices $\phi \colon N' \to N$ which send cones of $\Delta'$ into cones of $\Delta$.
Various algebraic-geometric properties of the morphism can be
understood in terms of conditions on the associated map of lattices.
In particular, properness is equivalent to the condition that
$\phi^{-1}(|\Delta|) =
|\Delta'|$.  If $\phi$ is the
identity map, then we say $\Delta'$ is a
$\tbf{refinement}$ of $\Delta$.  In what follows, we will be
interested in the inverse image of a closed subscheme made up of orbit
closures, especially in the case that $\Delta'$ is a refinement of
$\Delta$.
	
\begin{thm} \label{thm.inverseimage}\cite[Lemma 3.3.21 and Theorem 11.1.10]{CLS:11} With
  notation as above, suppose $\sigma$ is a
cone in $\Delta$ with orbit closure $V(\sigma) \subset X(\Delta)$.
If $\sigma_1,...,\sigma_k$ denote the minimal cones of $\Delta'$
such that $\sigma$ is minimal over $\phi(\sigma_i)$ then the irreducible decomposition of the inverse image of
$V(\sigma)$ is:
\[ \varphi^{-1}(V(\sigma)) = \bigcup^k_{i=1} V(\sigma_i).\]
In particular, in the case of a
refinement, $\sigma_i$ can be
characterized as the minimal cones
intersecting the relative interior of
$\sigma$.
\end{thm}
		
There is a special refinement which will be important in the remainder
of the paper - the star subdivision.  Geometrically, this corresponds
to the blowup of $X(\Delta)$ at a distinguished point
$\gamma_{\sigma}$.  Specifically, given a cone
$\sigma$, let $u_1,...,u_k$ be the primitive generators of the rays of
$\sigma$, then the $\tbf{star}$ of $\sigma$ is the vector
$\Star(\sigma) = \sum^k_{i=1} u_i$.  The $\tbf{star subdivision}$ of
$\sigma$, see \cite[Definition 3.3.13]{CLS:11}, is the fan denoted
$\sigma^*$ whose maximal cones are all of the form
$\cone(\mu,\Star(\sigma))$ where $\mu \subset \sigma$ is a facet.
	
	As a final remark, recall that a cone is $\tbf{smooth}$ if the
        minimal generators of the rays form part of an integral basis
        of $N$, and that a cone is $\tbf{simplicial}$ if the minimal
        generators are linearly independent over $\RR$.  A fan is
        smooth or simplicial if each of its cones are smooth or
        simplicial respectively.  It is well known that the toric
        variety $X(\Delta)$ is smooth if and only if each cone in
        $\Delta$ is smooth.
	
\subsection{Cox Construction of a Toric Variety}
The Cox construction produces any toric variety as the good quotient
of an open subscheme of affine space by the action of a diagonalizable
group.  Following \cite[Chapter 5]{CLS:11}, we briefly recall the
salient details of the construction.  Given a full-dimensional fan
$\Delta$ on a lattice $N$, the total coordinate ring of $X(\Delta)$ is
the polynomial ring $k[x_{\rho}: \rho \in \Delta(1)]$ freely generated
by the variables $x_{\rho}$. Suppose $n = |\Delta(1)|$ such that $\A^n
= \Spec(k[x_{\rho}])$, then for each cone $\sigma \in \Delta$, define
the monomial:
\begin{center}
$\displaystyle x^{\sigma} = \prod x_{\rho_i}$ where $\{\rho_i\}$ are the rays not contained in $\sigma$.
\end{center}
Furthermore, define $B(\Delta) = (x^{\sigma})_{\sigma \in \Delta}$ to be the ideal generated by these monomials.  The exceptional set of $\Delta$ is the following closed subscheme of $\A^n$:
\begin{center}
$\displaystyle Z(\Delta) = V(B(\Delta)) \subset \A^n$.
\end{center}

The closed set $Z(\Delta)$ has an explicit, combinatorial description. Recall that a $\tbf{primitive}$ $\tbf{collection}$ of rays is a collection of rays
satisfies the following two properties: the collection is not
contained in a single cone of $\Delta$, but every proper subset is
contained in a cone of $\Delta$.  For each primitive collection of
rays $\{\rho_i\}$, we can generate a linear subspace given by
$V((x_{\rho_i})_{i\in I})$.

\begin{prop}\cite[Proposition 5.1.6]{CLS:11} The irreducible components of $Z(\Delta)$ are in bijection with
  the linear subspaces $\VV((x_{\rho_i})_{i\in I})$ where $I$ is a
  primitive collection of rays.
\end{prop}

Given a fan $\Delta$, let $\Sigma_{\Delta}$ denote the fan generated by cones
$\displaystyle \ol{\sigma} = \cone(e_{\rho}: \rho \in \sigma(1)) \subset \RR^{|\Delta(1)|}$ where $\sigma$ runs through the cones of $\Delta$.

\begin{prop}\cite[Proposition 5.1.9]{CLS:11}
  (a)  $X(\Sigma_{\Delta}) = \A^n \setminus  Z(\Delta)$.

  (b) The map $e_\rho \mapsto u_\rho$ defines a map of lattices
  $\ZZ^{\Sigma(1)} \to N$ that is compatible with the fans $\Sigma_\Delta
  \subset \RR^{\Delta(1)}$ and $\Delta \subset N_{\RR}$.
\end{prop}

\begin{rem} \label{rem.orbitclosures}
  Note that if $\sigma$ is a cone of $\Delta$, then the orbit
  closure $V(\sigma)$ is the image of the linear subspace
  $\VV(\{x_\rho\}_{\rho \in \sigma(1)})$.
  \end{rem}

The diagonalizable group $G$ is
constructed as follows: suppose $\dim(\Delta) = k$, then the morphism
of lattices $\ZZ^n\rightarrow \ZZ^k$ defined by sending the canonical
basis vectors to the minimal generators of rays of $\Delta$ has finite
cokernel; hence, the dual map is injective, and we set $L$ to be the
cokernel of this map.  The group $G$ is defined to be $\Hom(L,\GG_m)$,
the group of characters of $L$, which injects into the torus of
$\A^n$.

\begin{thm}
  \cite[Theorem 5.1.11(a)]{CLS:11} Under the above notation,
  the toric map\\ $\left(\A^n \setminus  Z(\Delta)\right) \to X(\Delta)$
  identifies $X(\Delta)$ as the good quotient $\left(\A^n \setminus Z(\Delta)\right)/G$.
\end{thm}

\subsection{Canonical Toric Stacks}
Given a full-dimensional fan $\Delta$ on a lattice $N$ and its
associated toric variety $X(\Delta)$, the canonical toric stack is a
smooth quotient stack possessing a special relationship with
$X(\Delta)$.  Following \cite[Section 5]{GeSa:15a}, there are two
equivalent definitions for the canonical stack.  The
first is given in terms of the Cox construction, and the second
constructs $\mc{X}(\Delta)$ as a toric stack arising from a canonical
stacky fan.  The theory of stacky fans is worked out in \cite{GeSa:15a}
and provides a combinatorial gadget for toric stacks analogous
to fans for toric varieties. A stacky fan is a pair $(\Delta,
\beta:N\rightarrow L)$ where $\Delta$ is a lattice on $N$ and $\beta$
is a morphism of lattices with finite cokernel, and $\beta$ induces a
surjective morphism of tori $T_N\rightarrow T_L$ with kernel
$G_{\beta}$.  The toric stack associated to this data is 
the quotient stack $[X(\Delta)/G_{\beta}]$ with torus $T_L$.

Following \cite[Section 3]{GeSa:15a}, toric morphisms between toric
stacks arising from stacky fans $(\Delta', N'\rightarrow L'), (\Delta,
N\rightarrow L)$ correspond to morphisms of lattices which are
compatible with the stacky fans: specifically, 
$\phi:N'\rightarrow N$ and $\Phi:L'\rightarrow L$ are morphisms of
lattices such that for every cone $\sigma \in \Delta'$, $\phi(\sigma)$
is contained in a cone of $\Delta$ and the following diagram commutes:

\[
\xymatrix{
N' \ar[r]^\phi \ar[d]^{\beta'} & N \ar[d]^\beta \\
L' \ar[r]^{\Phi}
& L
}
\]

\begin{defn}
The canonical stack of $X(\Delta)$ will be denoted by $\mc{X}(\Delta)$ and is equipped with a toric morphism $\mc{X}(\Delta) \rightarrow X(\Delta)$; it can be constructed in either of the following ways:
\begin{enumerate}
\item The Cox construction of a toric variety produces an open subscheme $U\subset \A^{|\Delta(1)|}$ and a diagonalizable group $G$ such that $X(\Delta) \cong U/G$, \cite[Theorem 5.1.11]{CLS:11}.  The canonical stack of $X(\Delta)$ is the quotient stack $[U/G]$ with morphism $[U/G]\rightarrow X(\Delta)$.
\item Equivalently, let $(\Sigma_{\Delta},\beta:\ZZ^{|\Delta(1)|}\rightarrow N)$ be the stacky fan where $\beta$ is defined by sending the canonical basis vectors to the primitive generators of $\Delta$, and $\Sigma_{\Delta}$ is the canonical fan defined in the previous section.  The toric stack associated to this stacky fan is the canonical stack $\mc{X}(\Delta)$.  

Furthermore, the toric morphism $\mc{X}(\Delta) \rightarrow X(\Delta)$ is easily constructed from the following morphism of stacky fans:
\begin{center}
\begin{tikzcd}
\ZZ^{|\Delta(1)|} \arrow[r] \arrow[d]
& N \arrow[d, equal] \\
N \arrow[r, equal]
& N
\end{tikzcd}
\end{center}
\end{enumerate}
\end{defn}

Recall  \cite[Definition 4.1]{Alp:13}, that a morphism
from an Artin stack to an algebraic space is a $\tbf{good moduli space
  morphism}$ if it is cohomologically affine and Stein.  In the
context of quotient stacks, this is a natural generalization of good
quotients.  As a consequence of \cite[Theorem 6.3]{GeSa:15a}, the
following holds:

\begin{thm}
\cite[Example 6.23]{GeSa:15a} The canonical stack morphism $\mc{X}(\Delta) \rightarrow X(\Delta)$ is a good moduli space morphism.
\end{thm}

\begin{exa} \label{ex.basiccone}
Let $\sigma$ denote the cone generated by
$\{(1,0,1),(0,-1,1),(-1,0,1),(0,1,1)\}$ in $\RR^3$, then $X(\sigma)$
is a 3-dimensional, affine, singular toric variety with four rays.
Let $\beta:\ZZ^4 \rightarrow \ZZ^3$ be defined by sending the
canonical basis vectors to the primitive generators of $\sigma$; note,
the preimage of $\sigma$ under $\beta$ is the cone defining $\A^4$.
Furthermore, the cokernel of the dual of $\beta$ is $\ZZ^4\rightarrow
\ZZ^3$ defined as the matrix $[1,-1,1,-1]$; hence, if we apply $D(-)
:= \Hom(-,\GG_m)$ to this cokernel, we see that $G_{\beta} = \GG_m$
and defines an action of $\A^4$ with weights $(1,-1,1,-1)$.  Thus, the
canonical stack of $X(\sigma)$ is $\mc{X}(\sigma) = [\A^4/\GG_m]$.
\end{exa}

\section{Chow Cohomology of Algebraic Spaces and Stacks}

Following \cite{Ful:84}, let $A_k(X)$ denote the group of dimension
$k$ cycle classes modulo rational equivalence, and if $X$ is
equidimensional, let $A^k(X)$ denote the group of codimension $k$
cycle classes modulo rational equivalence.  In addition if $X$ is
smooth, then the intersection product on $A_k(X)$ as constructed in
\cite[Chapter 6.1]{Ful:84} makes $A^*(X)$ into a commutative, graded
ring.  For general schemes, without any assumptions on
non-singularity, \cite[Chapter 17]{Ful:84} constructs a graded ring
$A^*_{\op}(X) := \oplus_{k \geq 0} A^k_{\op}(X)$ defined as follows: an
element $c\in A^k_{\op}(X)$ is a collection of homomorphisms of groups:
\[ c^{(k)}_g: A_pX'  \rightarrow A_{p-k}X'\]
for every morphism $g:X'\rightarrow X$  which are
compatible with respect to proper pushforward, and flat and
l.c.i. pullbacks (see \cite[Definition's 17.1 and 17.3]{Ful:84}).  The
product is given by composition and turns $A^*_{\op}(X)$ into
a unital, associative, graded ring called the $\tbf{Chow
  cohomology ring}$ of $X$. Moreover, if
$X$ has a resolution of singularities (e.g. if $X$ is a toric variety) then
$A^*_{\op}(X)$ is known to be commutative.
If $X$ is smooth, then \cite[Corollary
  17.4]{Ful:84} proves that the Poincar\'e-Duality map $A^k_{\op}(X)
\rightarrow A^k(X) = A_{n-k}(X)$ is an isomorphism of rings wherein
the intersection product agrees with the product given by composition.

\begin{rem}
In \cite[Chapter 17]{Ful:84} the Chow cohomology
ring is also denoted $A^*(X)$ without the inclusion of the subscript
'op', regardless if $X$ is smooth or not.  We will deviate from this
convention to emphasize the distinction between operational classes in
the smooth and non-smooth cases.
\end{rem}

For smooth quotient stacks $\mc{X} = [Z/G]$, the Chow ring $A^*(\mc{X})$
is identified with the $G$-equivariant Chow ring $A^*_G(Z)$ by \cite{EdGr:98}.
Furthermore, if $\pi:\mc{X} \rightarrow X$
is any morphism between a stack and an algebraic space, \cite[Proposition
  2.10]{EdSa:17} proves that there is a
a pullback map $\pi^*:A^*_{\op}(X) \rightarrow A^*(\mc{X})$, $c \mapsto c \cap [\mc{X}]$.  In this paper we make essential use of the following result.
\begin{thm}
\cite[Theorem 1.1]{EdSa:17} Let $\mc{X}$ be a smooth, connected, properly stable, Artin stack with good moduli space $\pi:\mc{X}\rightarrow X$, then $\pi^*$ is injective over $\QQ$.
\end{thm}

Since $\mc{X}(\Delta) \rightarrow
X(\Delta)$ is a good moduli space morphism, we conclude
that $A^*_{\op}(X(\Delta))_{\QQ}$ injects into
$A^*(\mc{X}(\Delta)_{\QQ}$.  If $\Delta$ is simplicial, then $\mc{X}(\Delta)$ is Deligne-Mumford and the inclusion
is an isomorphism.

\subsection{Chow Cohomology of Toric Varieties}

For toric varieties, the generators and relations of the Chow groups $A_k(X(\Delta))$ can be explicitly computed.  By \cite[Chapter 5.1]{Ful:93}, $A_k(X(\Delta))$ is generated by the classes $[V(\sigma)]$ where $\sigma$ has codimension $k$.  Furthermore, the group of relations between these generators is generated by all relations of the form:
\begin{center}
$\displaystyle [\divi(x^u)] = \sum_{\sigma} \left\langle u,n_{\sigma,\tau} \right\rangle [V(\sigma)]$,
\end{center}
where $\tau$ runs over all cones of codimension $k+1$ which are a face
of $\sigma$, $u$ runs over a generating set of $M(\tau)$, and
$n_{\sigma,\tau}$ is a rational point whose image generates the
lattice $N_{\sigma}/N_{\tau}$.  In the case that $X(\Delta)$ is
smooth, the above description can be used to compute
$A^*_{\op}(X(\Delta))$ since $A^k_{\op}(X(\Delta)) =
A_{n-k}(X(\Delta)$.  However when $X(\Delta)$ is not smooth, there are
only a few known techniques to compute $A^*_{\op}(X(\Delta))$.  In
particular, for complete toric varieties, the Chow cohomology ring can
also be derived from the above description of Chow groups by the
following theorem:

\begin{thm}
\cite[Proposition 1.4, Theorem 2.1]{FuSt:97} If $X$ is a complete toric variety, then 
$A^k_{\op}(X) \cong \Hom(A_k(X),\ZZ)$,
and the Chow cohomology is also isomorphic to the ring of Minkowski weights on $\Delta$.
\end{thm}

However, for non-complete toric varieties this isomorphism fails as the following example demonstrates.

\begin{exa} \label{exa.opnothom}
  Let $\sigma$ denote the cone generated by $\{(1,0,1),(0,-1,1),(-1,0,1),(0,1,1)\}$ in $\RR^3$ as seen above.  One can compute that $A_2(X(\sigma)) = \ZZ/2 \oplus \ZZ$ which does admit non-trivial morphisms $A_2(X(\sigma)) \rightarrow \ZZ$.  However, by Theorem \ref{thm.Chowcoh}, $A^2_{\op}(X(\sigma)) = 0$. 
\end{exa}

\subsubsection{Computing Chow Cohomology: Kimura's Exact Sequence}

There is a second technique to compute the Chow cohomology of a toric
variety using Kimura's exact sequence.  This sequences injects
$A^*_{\op}(X(\Delta))$ into the Chow cohomology of any envelope of
$X(\Delta)$.  Recall, a morphism $X'\rightarrow X$ is an envelope if
it is proper and for every closed subvariety $V\subset X$, there
exists a closed subvariety $V' \subset X'$ mapping onto $V$ such that
the restriction $V'\rightarrow V$ is birational.

\begin{prop}\cite[Theorem 2.3]{Kim:92} Let $X'\rightarrow X$ be an envelope which is an isomorphism outside of $S\subset X$ and $E \subset X'$, then

  \begin{equation} \label{eq.kimuraseq}
0\rightarrow A^*_{\op}(X) \rightarrow A^*_{\op}(X') \oplus A^*_{\op}(S) \rightarrow A^*_{\op}(E)
\end{equation}
is exact.
\end{prop}

If $X' \to X$ is proper and surjective, then the above sequence is
exact after tensoring with $\QQ$.  Payne proved 
\cite[Lemma 1]{Pay:06} that any proper 
birational, toric morphism is an envelope.  Hence, if
$\Delta'$ is any refinement of $\Delta$, then $A^*_{\op}(X(\Delta))$
injects into $A^*_{\op}(X(\Delta'))$.  We make  essential use of this fact
in our proof of Theorem \ref{thm.Chowcoh}

\subsection{Chow Rings of Canonical Toric Stacks}
The Chow ring of $\mc{X}(\Delta)$ admits a simple presentation
using equivariant Chow groups.
Following \cite[Section 2]{EdCo:18}, let $n =
|\Delta(1)|$, $V = \A^{n}$, $U = V \setminus  Z(\Delta)$, $G$ the
diagonalizable group such that $\mc{X}(\Delta) = [U/G]$ and $X(G)$ the
character group of $G$.  The excision exact sequence
(\cite[Prop. 1.8]{Ful:84}) for equivariant Chow groups implies
$A^*(\mc{X}(\Delta))$ is a quotient of $A^*_G(V) = A^*_G(pt)$; the
latter $A^*_G(pt)$ will be denoted by $A^*_G$.  By \cite{EdGr:98},
$A^*_G \cong \ZZ[X(G)]$ where the latter is the symmetric algebra of
the character group of $G$.  By construction of $\mc{X}(\Delta)$, the
action of $G$ on $V$ is faithful; hence, $G$ injects into $\GG^n_m$
and induces a surjection of character groups $X(\GG^n_m) \rightarrow
X(G)$. This surjection further induces a surjection of equivariant
Chow rings $A^*_{\GG^n_m} \rightarrow A^*_G$ leading us to the
following definition:

\begin{defn}
The $\tbf{linear equivalence ideal}$ of $\Delta$ is defined to be:
\begin{center}
$\displaystyle L(\Delta) := \ker(\A^*_{\GG^n_m} \rightarrow A^*_G)$.
\end{center}
Equivalently, it is the linear ideal generated by the relations among the image of $X(\GG^n_m)$ in $X(G)$.  Hence, $A^*_G = A^*_{\GG^n_m}/L(\Delta)$.
\end{defn}

Suppose $L_1,...,L_k$ are the irreducible components of $Z(\Delta)$,
then by
\cite[Propositions 2.1 and 2.2]{EdCo:18}, $A^*_G(U)$ is the quotient
of $A^*_{\GG^n_m}$ modulo the linear equivalence ideal and the ideal
generated by the equivariant fundamental classes of $L_i$.

\begin{defn}
The $\tbf{Stanley-Reisner ideal}$ of $\Delta$ is the ideal $\mc{Z}(\Delta)$ generated by the equivariant fundamental classes of the irreducible components of $Z(\Delta)$.
\end{defn}

Thus, by \cite[Propositions 2.1 and 2.2]{EdCo:18}, the Chow ring of $\mc{X}(\Delta)$ has the following presentation:

\begin{thm} \label{thm.chows-r}
Under the above notation,
\begin{center}
$\displaystyle A^*(\mc{X}(\Delta)) \cong A^*_G/\mc{Z}(\Delta) \cong A^*_{\GG^n_m}/(L(\Delta),\mc{Z}(\Delta))$.
\end{center}
\end{thm}

Each of the terms in the above presentation can be explicitly
computed.  Let $e_1,...,e_n$ denote the canonical basis characters in
$X(\GG^n_m)$, and let $t_1,...,t_n$ denote the corresponding first
Chern classes of the associated 1-dimensional representations of each
character, then $A^*_{\GG^n_m} = \ZZ[t_1,...,t_n]$.  For the linear
equivalence ideal, let $v_1,...,v_k$ be a basis for $\ker(X(\GG^n_m)
\rightarrow X(G))$, then $L(\Delta)$ is the ideal generated by
corresponding first Chern classes of $v_1,...,v_k$: if $v_i = \sum_j
a_{ij}e_j$, then $\sum_j a_{ij}t_j \in L(\Delta)$.  For the
Stanley-Reisner ideal, suppose $L_i = V(x_{i_1},...,x_{i_l})$ where
$x_1,...,x_n$ are coordinates for $V$, then the equivariant
fundamental class of $L_i$ is $t_{i_1}\cdot ... \cdot t_{i_l}$.

\begin{rem}
Since we know the irreducible components of $Z(\Delta)$ correspond to
the primitive collections of $\Delta$, $\mc{Z}(\Delta)$
is equivalently generated by the monomials corresponding to primitive
collections.  Explicitly, if $\{ \rho_{i_1},...,
\rho_{i_k} \}$ is a primitive collection corresponding to an irreducible
component of $Z(\Delta)$, 
then the monomial $t_{i_1}
\cdot ... \cdot t_{i_k}$ is the
equivariant fundamental class of this component.
\end{rem}

\begin{exa}
Let $\sigma$ be the cone in the previous examples, then since
$X(\sigma)$ is affine, $Z(\sigma) = 0$.  Furthermore, $L(\sigma)$ is
generated by the relations $(s_1-s_3,s_2-s_4,s_1+...+s_4)$.  Hence,
$A^*(\mc{X}(\sigma))$ is a polynomial ring in one variable.
\end{exa}

\begin{exa} \label{exa.condivided}
Suppose we star subdivide $\sigma$ to produce the fan $\sigma^*$.  The
Chow cohomology of the canonical stack of $\sigma^*$ has the following
presentation: the Stanley-Reisener ideal is the ideal
$(s_1s_3,s_2s_4)$ since $\{\rho_1, \rho_3\}$ and $\{\rho_2, \rho_4\}$
are the primitive collections of $\sigma^*$.  Furthermore, the linear
equivalence ideal is defined by the equations
$(s_1,s_3,s_2-s_4,s_1+...+s_5)$.  Hence, $A^*(\mc{X}(\sigma^*)) =
\ZZ[s_1,s_2]/(s_1^2,s_2^2)$.
\end{exa}

\section{Chow Cohomology of Affine Toric Varieties}

In this section, we will establish our main result, the vanishing of
the Chow cohomology of an affine toric variety.  

\begin{thm} \label{thm.Chowcoh}
For any affine toric variety $X(\sigma)$,
\begin{center}
$\displaystyle A^k_{\op}(X(\sigma)) = 0$ for $k > 0$.
\end{center}
In particular, $A^*_{\op}(X(\sigma)) = \ZZ$.
\end{thm}

To slightly simplify the notation we will assume that the cone is full
dimensional. If $\sigma$ is not full dimensional then $X(\sigma) = T
\times X(\overline{\sigma})$ where $T$ is a torus and
$\overline{\sigma}$ is a full-dimensional cone and our
full-dimensional proof readily adapts to this case.

The proof will be divided into the following subsections.

\subsection{Kimura's Exact Sequence}

Let $\sigma^*$ denote the star subdivision of $\sigma$ with respect to
its star, $v = \Star(\sigma)$, $\rho_v$ the corresponding ray in
$\sigma^*$, and $\phi:X(\sigma^*) \rightarrow X(\sigma)$ the
associated toric morphism.  Note, this morphism is an isomorphism
outside of $S:= V(\sigma)$ and $E := \phi^{-1}(S)$.

\begin{lem}
Under the above notation, $E = V(\rho_v)$.
\end{lem}
\begin{proof}
By construction, $\rho_v$ is contained in the interior of $\sigma$,
and if $\tau$ is a cone of $\sigma^*$ non-trivially intersecting the
interior of $\sigma$, then $\tau$ necessarily contains $\rho_v$.
Hence, $\rho_v$ is the unique minimal cone of $\sigma^*$ intersecting
the interior of $\sigma$.  The Lemma then follows from the description of
Theorem \ref{thm.inverseimage}.
\end{proof}

Furthermore, $S$ is $0$-dimensional and irreducible, hence, $A^0(S) = \ZZ$ and $A^k(S) = 0$ for $k > 0$.  Therefore, if we apply Kimura's exact sequence to $\phi$, we obtain the following exact sequence:

\begin{center}
$\displaystyle 0 \rightarrow A^k_{\op}(X(\sigma)) \rightarrow A^k_{\op}(X(\sigma^*)) \rightarrow A^*_{\op}(V(\rho_v))$ for $k > 0$.
\end{center}

Thus, to show $A^k_{\op}(X(\sigma)) = 0$ for $k > 0$, it suffices to
show $A^k_{\op}(X(\sigma^*)) \rightarrow A^*_{\op}(V(\rho_v))$ is
injective.  To accomplish this, Steps 2 and 3 will further demonstrate
that it suffices to consider the Chow rings of the associated
canonical stacks. In Step 4 we will show that these Chow rings are in fact isomorphic.

\subsection{A Commutative Diagram of Canonical Stacks} \label{sec.excep}

Let $\mc{X}(\sigma^*) \rightarrow X(\sigma^*)$ and $\mc{E} \rightarrow E$ denote the canonical stacks of $X(\sigma^*)$ and $E$ respectively, and let $E \rightarrow X(\sigma^*)$ denote the inclusion, then the goal of this subsection is to show that we can construct a closed embedding $\mc{E} \rightarrow \mc{X}(\sigma^*)$ such that the following diagram of canonical stacks and moduli spaces is commutative:

\begin{equation}\label{diag.canonical}
\begin{tikzcd} 
\mc{E} \arrow[r] \arrow[d] & \mc{X}(\sigma^*) \arrow[d] \\
E \arrow[r] & X(\sigma^*)
\end{tikzcd}
\end{equation}

Let $\Sigma_{\sigma^*}$ denote the fan of
$\A^{|\sigma^*(1)|} \setminus Z(\sigma^*)$. It is equipped with a
morphism $X(\Sigma_{\sigma^*}) \rightarrow X(\sigma^*)$ as in the
canonical stack construction.  If $n = |\sigma(1)|$, let
$x_1,...,x_{n+1}$ be coordinates on $\A^{|\sigma^*(1)|}$ such that
$x_{n+1}$ corresponds to the ray $\rho_v$.  By Remark
\ref{rem.orbitclosures} we have a commutative diagram of $G$-quotients

\begin{equation} \label{diag.cart}
\begin{tikzcd}
(\VV(x_{n+1})\setminus Z) \arrow[r] \arrow[d]
& X(\Sigma_{\sigma}) \arrow[d] \\
E \arrow[r]
& X(\sigma^*)
\end{tikzcd}
\end{equation}
where $G$ is the group acting on $X(\Sigma_{\sigma^*})$ such that
$\mc{X}(\sigma^*) = [X(\Sigma_{\sigma^*})/G]$. We will
show that $[(\VV(x_{n+1}) \setminus Z/G)]$ is the canonical stack
of the toric variety $E$.

We begin
with the following basic result of toric varieties:

\begin{lem}
A collection of rays of $\sigma^*$ is primitive if and only if the corresponding collection in $\Star(\rho_v)$ is primitive.
\end{lem}
\begin{proof}
Let $C$ be a primitive collection of $\sigma^*$ and and let $C'$ be
the associated collection in $\Star(\rho_v)$.  If $D' \subset C'$ is a
proper subset and $D \subset C$ denotes the corresponding subset, then
since $D$ is contained in some cone, it follows that $D'$ will be
contained in the image of that cone, and similarly, if $C' \subset
\ol{\tau}(1)$, then $C \subset \tau$ would contradict $C$ being
primitive.  Thus, if $C$ is primitive, then $C'$ is primitive; the
converse holds by a similar argument.
\end{proof}
The existence of the commutative diagram \eqref{diag.canonical} now follows from the following lemma.
\begin{lem}
Let $\mc{E}$ denote the canonical stack over $E$, then $\mc{E} \cong [(\VV(x_{n+1})\setminus Z)/G]$.
\end{lem}

\begin{proof}
By definition, $\mc{E}$ is the quotient stack of the form
$[(\A^n\setminus Z')/G']$ where $Z'$ and $G'$ are defined as
previously.  By the previous lemma, there is an obvious isomorphism
between $\A^n\setminus Z'$ and $\VV(x_{n+1})\setminus Z(\sigma^*)$.  We
claim that $G = G'$ and the action of $G'$ on $\A^n\setminus Z'$
coincides with the action of $G$ on $\VV(x_{n+1})\setminus
Z(\sigma^*)$ under this isomorphism.  Indeed, this follows from the
commutativity of the following diagram of character groups:

\begin{center}
\begin{tikzcd}
0 \arrow[r] & \ZZ^{k-1} \arrow[r] \arrow[d] & \ZZ^n \arrow[r] \arrow[d] & X(G') \arrow[r] \arrow[d,equal] & 0
\\
0 \arrow[r] & \ZZ^k \arrow[r] & \ZZ^{n+1} \arrow[r] & X(G) \arrow[r] & 0
\end{tikzcd}
\end{center}

The vertical maps are the duals of the canonical quotient maps for the primitive generators of $\rho_v$ and $e_{n+1}$ respectively.
\end{proof}

\begin{rem}
  Note that unless $\sigma^*$ is a resolution of singularities,
  diagram \eqref{diag.cart} is not cartesian.  The following examples shows that even if $\sigma^*$ is a simplicialization the diagram fails to be Cartesian.
\end{rem}

\begin{exa}
 Consider the cone $\sigma$ with ray generators given by $\{(1,1,1)$, $(-1,0,1)$, $(1,0,1)$, $(0,-1,1)\}$, and its associated star subdivision $\sigma^*$ obtained by adding the ray generator $(1,0,4)$ and subdividing appropriately.  We claim that in this example, $\mc{E}$ is not scheme-theoretically saturated with respect to $\pi$.  Following \cite[Remark 3.4]{EdSa:17}, since $\dim(\mc{E}) = \dim(E)$ and $\pi(\mc{E}) = E$, it suffices to show $\mc{E}$ is not a strong divisor.

With the notation as in this section, the action of $\GG_m^2$ on $\A^5 \setminus V(x_1,x_3)\cup V(x_2,x_4)$ is given by the matrix:

\begin{center}
$\left[
\begin{array}{ccccc}
	3 & -2 & 1 & -2 & 0 \\
	2 & -3 & 0 & -3 & 1
\end{array}
\right] $
\end{center}

If $\mc{E}$ is strong, then the ideal $(x_5)$ is generated by an
invariant function on each open set $D(x_1x_2)$, $D(x_2x_3)$,
$D(x_3x_4)$, and $D(x_1x_4)$.  However, on $D(x_1x_2)$ where $x_1,x_2$
are invertible, the weight $(0, 1)$ of $x_5$ cannot be expressed
as integral linear combination of the weights $(3,2)$ and $(-2,-3)$ for $x_1,x_2$
respectively; thus, $\mc{E}$ is not strong. Note that $(0,1)$ can be expressed rationally in terms of
these weights on each affine open set and it is easy
to check that $(x_5^{15})$ is locally generated by an invariant function.

\end{exa}

\subsection{The injectivity result of \cite{EdSa:17}}

Since the canonical stack morphisms are good moduli space morphisms, we have that $A^*_{\op}(E)_{\QQ} \subset A^*(\mc{E})_{\QQ}$ and $A^*_{\op}(X(\sigma^*))_{\QQ} \subset A^*(\mc{X}(\sigma^*)_{\QQ}$.  Hence, by combining these injections with the main result of the previous section and Kimura's exact sequence, the following diagram commutes:

\begin{center}
\begin{tikzcd}
0 \arrow[r] & A^k_{\op}(X(\sigma))_{\QQ} \arrow[r] & A^k_{\op}(X(\sigma^*))_{\QQ} \arrow[r] \arrow[d] & A^k_{\op}(E)_{\QQ} \arrow[d] \\
&  & A^k(\mc{X}(\sigma^*))_{\QQ} \arrow[r]    & A^k(\mc{E})_{\QQ}
\end{tikzcd}
\end{center}

Since the vertical maps are injections, to show $A^k_{\op}(X(\sigma^*)) \rightarrow A^k_{\op}(E)$ is injective at least over $\QQ$, it suffices to show $A^*(\mc{X}(\sigma^*))_{\QQ} \rightarrow A^*(\mc{E})_{\QQ}$ is an injection.  In fact, we claim this map is an isomorphism over $\ZZ$.

\subsection{The Pullback of Chow Rings of Canonical Stacks is an Isomorphism}

\begin{thm} \label{thm.epullback}
Let $\mc{E} \rightarrow \mc{X}(\sigma^*)$ be the injective morphism constructed in Step 2, then the pullback induces an isomorphism $A^*(\mc{E}) \cong A^*(\mc{X}(\sigma^*))$.
\end{thm}

\begin{proof}
By the previous lemma, $\mc{E} \cong [(\VV(x_{n+1})\setminus Z(\sigma^*))/G] \subset \mc{X}(\sigma^*) = [(\A^{n+1}\setminus Z(\sigma^*))/G]$.  Hence, the Chow rings of the canonical stacks are the following equivariant Chow rings:
\begin{center}
$\displaystyle A^*(\mc{X}(\sigma^*)) = A^*_G(\A^{n+1}\setminus Z(\sigma^*))$, and $\displaystyle A^*(\mc{E}) = A^*_G(V(x_{n+1})\setminus Z(\sigma^*))$.
\end{center}
Since $Z(\sigma^*)$ does not contain $x_{n+1}$, we can write
$\VV(x_{n+1})\setminus Z(\sigma^*) = \A^n \setminus  Z'$ and $\A^{n+1}\setminus Z(\sigma^*) =
(\A^n\setminus Z')\times \A^1$ for the obvious exceptional set $Z' \subset
\A^n$.  In particular, the inclusion of canonical stacks is the quotient
by $G$ of the zero section
$\A^n\setminus Z' \subset (\A^n\setminus Z')\times \A^1$ of a
$G$-equivariant vector bundle over $\A^n \setminus Z'$.
Hence, the
pullback in $G$-equivariant Chow groups along this inclusion is an isomorphism.
\end{proof}

Thus, $A^k_{\op}(X(\sigma))_{\QQ} = 0$ for $k > 0$.  The final step will be to show $A^k_{\op}(X(\sigma))$ is torsion-free.

\subsection{Chow Cohomology of Semi-Projective Toric Varieties is Torsion-Free}

From \cite{HaSt:02}, a toric variety $X(\Delta)$ is semi-projective if
it has at least one torus fixed point and the natural map
$X(\Delta)\rightarrow \Spec(\Gamma(\mc{O}_{X(\Delta)}))$ is
projective; equivalently, it is semi-projective if $\Delta$ has
full-dimensional convex support and $X(\Delta)$ is quasi-projective.
Hence, for an affine toric variety $X(\sigma)$, any refinement
$\Delta$ of $\sigma$ produces a semi-projective toric variety
$X(\Delta)$ which is projective over $X(\sigma)$.  In particular, let
$\sigma'$ be a resolution of singularities of $\sigma$, then by
\cite[Lemma 2.1]{Kim:92}, $A^*_{\op}(X(\sigma)) \subset
A^*_{\op}(X(\sigma'))$.  Hence, it suffices to show
$A^*_{\op}(X(\sigma'))$ is torsion-free where $X(\sigma')$ is a smooth
semi-projective toric variety.

\begin{rem}
 For a smooth projective toric variety, \cite[Section 5.2]{Ful:93}
 demonstrates that its Chow ring is torsion-free by producing
 Bialynicki-Birula decomposition from a particular ordering of the
 maximal dimensional cones.
\end{rem}

In \cite[Section 2]{HaSt:02}, Hausel and Sturmfels generalize the
construction of \cite[Section 5.2]{Ful:93} to smooth semi-projective
toric varieties by utilizing the moment map associated to a
1-parameter subgroup, and in particular, they construct the
Bialynicki-Birula decomposition of a smooth semiprojective variety by
producing a collection of locally closed subsets $\{U_j\}$ satisfying
the following properties:

\begin{enumerate}
\item Each $U_i$ is a union of orbits, and hence, the closure of $U_i$ is a union of $U_j$'s.
\item $X(\sigma')$ is the disjoint union of the $U_i$'s.
\item Each $U_i \cong \A^{n-k_i}$.
\end{enumerate}

In particular, by \cite[Definition 1.16]{EiHa:16}, this decomposition
of $X(\sigma')$ is an affine stratification.  By \cite{Tot:14}, the
classes of $U_i$ form a basis for $A^*(X(\sigma'))$, and hence,
$A^*(X(\sigma'))$ is necessarily torsion-free.  Thus, since
$A^*_{\op}(X(\sigma)) \subset A^*(X(\sigma'))$, it follows that
$A^*_{\op}(X(\sigma))$ is torsion-free, and by combining all of the
previous subsections, this concludes the proof of the main theorem. \qed

The argument presented in this subsection can be utilized to demonstrate that the Chow cohomology rings of a wide class of toric varieties are torsion-free.  As a natural generalization of semi-projectivity, we say that a toric variety $X(\Delta)$ is semi-proper if it has at least one torus fixed point and the natural map $X(\Delta) \rightarrow \Spec(\Gamma(\mc{O}_{X(\Delta)}))$ is proper.

\begin{thm} \label{thm.chowtorfree}
Let $X(\Delta)$ be any semi-proper toric variety, then $A^*_{\op}(X(\Delta))$ is torsion-free.
\end{thm}
\begin{proof}
By the toric Chow's lemma \cite[Theorem 6.1.18]{CLS:11}, there exists
a projective toric variety $X(\Delta')$, and toric morphism
$X(\Delta') \rightarrow X(\Delta)$, where $\Delta'$ is a smooth
refinement of $\Delta$, such that the map $X(\Delta') \rightarrow
X(\sigma)$ factors through $X(\Delta) \rightarrow X(\sigma)$.  Hence,
by \cite[Lemma 2.1]{Kim:92}, $A^*_{\op}(X(\Delta)) \subset
A^*(X(\Delta))$ and the latter is torsion-free by the previous
arguments in this section.
\end{proof}

\begin{exa} Let $\sigma$ denote the cone generated by $\{(1,0,1),(0,-1,1),(-1,0,1),(0,1,1)\}$ in $\RR^3$, then $\sigma^*$ is actually a resolution of singularities of $\sigma$.  In particular, the vertical maps are isomorphisms in Kimura's sequence.  Furthermore, $E = V(\rho_v) = \PP^1\times \PP^1$, and we can explicitly compute $A^*(\mc{X}(\sigma^*))$ as:
\begin{center}
$\displaystyle A^*(\mc{X}(\sigma^*)) =
  \ZZ[s_1,...,s_4,s]/(s_1-s_3,s_2-s_4, s_1 + ... + s_4 +s, s_1s_3,
  s_2s_4)$.
\end{center}
Hence, $s_1 = s_3$, $s_2=s_4$, and $s = -2s_1 - 2s_2$; and $A^*(\mc{X}(\sigma^*)) = \ZZ[s_1,s_2]/(s_1^2,s_2^2)$ is clearly isomorphic to the Chow ring of $\PP^1 \times \PP^1$ as predicted by Theorem \ref{thm.epullback}.
This calculation also directly verifies that $A^*_{\op}(X(\sigma))$
is torsion free since it injects into the ring $\ZZ[s_1,s_2]/(s_1^2, s_2^2)$
which has no $\ZZ$-torsion.
        \end{exa}

\section{Operational $K$-theory}
In this part we prove Theorem \ref{thm.opKT}. The proof is
similar to the proof of Theorem \ref{thm.Chowcoh} so we give an
extended sketch. The main technical difficulty is establishing a
version of the injectivity result of \cite{EdSa:17} for operational
$K$-theory
\subsection{Notation and background on operational $K$-theory}
We briefly recall the notation and basic results on operational $K$-theory
defined by Anderson and Payne in the paper \cite{AnPa:15}.

Following \cite{AnPa:15}, if $X$ is a scheme (or more generally a stack)
we denote by $K_0(X)$ the Grothendieck group of coherent sheaves, and
$K^0(X)$ the Grothendieck group of perfect complexes. If $X$ has an ample
family of line bundles, then $K^0(X)$ is the same as the naive Grothendieck
group of vector bundles.

For any scheme $X$, Anderson and Payne define the {\bf operational
  $K$-theory} $\op K^0(X)$ of $X$ as follows.  An element $c \in \op
K^0(X)$ is a collection of operators $c_f \colon K_0(X') \to K_0(X')$
indexed by morphisms $X' \stackrel{f} \to X$ compatible with proper
pushforward, flat and lci pullback.

For any scheme $X$, there is a canonical map $\op K^0(X) \to K_0(X)$ given
by $c \mapsto c_{\id_X}({\mathcal O}_X)$. If $X$ is smooth, then
\cite[Corollary 4.5]{AnPa:15} states that this map is an isomorphism.
They also prove \cite[Proposition 5.3]{AnPa:15} that the groups $\op K^0(X)$
satisfy descent for Chow envelopes. In particular, the exact sequence
\eqref{eq.kimuraseq} also holds for operational $K$-theory.

\begin{thm} \label{thm.opKT}
  For any affine toric variety $X(\sigma)$, $\op K^0(X(\sigma) ) = \ZZ$.
  \end{thm}
\subsection{The injectivity theorem for operational $K$-theory}

\begin{lem} \label{lem.envpullback}
 If $X' \stackrel{f} \to X$ is an envelope (resp. proper surjective),
 then the pullback map on operational $K$-theory $f^* \colon \op K^0(X)
 \to \op K^0(X)$ is injective (resp. rationally injective).
 \end{lem}

\begin{proof} This follows from the formal properties of operational $K$-theory and the fact that the pushforward
map $f_*\colon K_0(X') \to K_0(X)$ is surjective (resp. rationally
surjective).
\end{proof}

The following is the $K$-theoretic analogue of
\cite[Proposition 2.10]{EdSa:17}.
\begin{prop} \label{prop.kpullback}
If $\X \to X$ is a good moduli space morphism and $X$ has a resolution
of singularities, then there is a pullback $\op K^0(X) \to K_0(\X)$,
$c \mapsto c({\mathcal O}_{\X'})$.  In general this map is
defined with $\QQ$ coefficients.
\end{prop}
\begin{proof}
Assuming $X$ admits a resolution of singularities, then there is a
birational proper morphism $f \colon X' \to X$ with $X'$ a smooth
scheme. Let $\X' = X' \times_X \X$ and let $f \colon \X' \to \X$ be
the morphism of stacks obtained by base change.  If $c \in \op K^0(X)$
then since $\op K^0(X') =K_0(X')=K^0(X)$ we can identify $f^*c$ with an
element of $K^0(X)$. Since pullbacks in $K$-theory of vector bundles
exist for arbitrary morphisms, we obtain a class $\pi'f^*c \in
K^0(\X')$. Using the same name for its image in $K_0(\X')$, we
define $\pi^*c$ to be $g_*\pi'f^*c \in K_0(\X)$. The projection
formula implies that this definition is independent of the choice of
resolution $X'$.
\end{proof}

We now obtain the $K$-theoretic analogue of the injectivity result
proved in \cite{EdSa:17}. 
\begin{thm} \label{thm.kpullbackinj}
  Let $\pi \colon \X \to X$ be a properly stable, good moduli space morphism
  with $\X$ a smooth Artin stack with generically trivial stabilizer and $X$ a scheme. Then the pullback $\pi^* \colon
  \op K^0(X) \to K_0(\X)$ is injective after tensoring with $\QQ$.
\end{thm}
\begin{proof}
  By \cite[Corollary 7.2]{EdRy:17}, there is a commutative diagram of 
  stacks and good moduli spaces
  $$\xymatrix{ \X' \ar[d]^{\pi'}  \ar[r]^g & \X \ar[d]^\pi\\ X' \ar[r]^f & X}$$ with
  $\X'$ a smooth tame stack with smooth coarse space $X'$ such that the map $f \colon X' \to X$ is proper
  and birational.  By Lemma \ref{lem.envpullback},
  we know that $f^*\colon \op K^0(X)
  \to \op K^0(X')$ is injective after tensoring with $\QQ$.
  Thus, it suffices to prove that $\pi'^* \colon \op K^0(X') \to K_0(\X')$
  is injective after
  tensoring with $\QQ$. In other words, we are reduced to the case that $\X$ is a smooth tame stack with smooth coarse moduli space $X$.
  In this case $\op K^0(X) = K_0(X) = K^0(X)$. 
  By the Riemann-Roch theorem for smooth schemes, the Chern
  character $\ch \colon K_0(X)_\QQ \simeq A^*(X)_\QQ$ is an isomorphism.
  The Chern character commutes with the pullback $\pi^* \colon K_0(X) \to K_0(\X)$, and the pullback map in Chow groups
  $\pi^* \colon A^*(X)_\QQ \to A^*(\X)_\QQ$ is an isomorphism.
  Hence, the map $\pi^*$ is necessarily injective\footnote{Note that the Chern
  character $\ch \colon K_0(\X) \to A^*(\X)$ is  defined because the hypothesis
  ensures that $\X$ is a quotient stack by \cite{EHKV:01}. The Chern character
  is surjective but not injective. See \cite{Edi:13} for a discussion of the Riemann-Roch theorem on tame stacks.}.
  It follows from functoriality of the Riemann-Roch map that the
  composition $K_0(X)_\QQ \to K_0(\X) \to A^*(\X)_\QQ$
  is rationally injective.
  \end{proof}
  
\subsection{$K$-theory of toric canonical stacks}
Like the Chow ring, the $K$-theory of a canonical toric stack $\mc{X}(\Delta)$ admits a
simple Stanley-Reisner type presentation.

The $\tbf{$K$-theoretic linear equivalence ideal}$ of $\Delta$ is defined to be:
\[ KL(\Delta) := \ker\left(R(\GG^n_m) \rightarrow R(G)\right).\]

If $e_1, \ldots, e_n$ is the canonical basis of characters, then
$R(\GG^n_m) = \ZZ[e_1,e_1^{-1}, \ldots , e_n, e_n^{-1}]$ and
$KL(\Delta)$ is generated by expressions of the form $e_1^{a_1} \ldots
e_n^{a_n}-1$ with $a_1 e_1 + \ldots + a_n e_n \in \ker \left( X(\GG_m^n)
\to X(G)\right)$.
\begin{defn}
  The $\tbf{$K$-theoretic Stanley-Reisner ideal}$ of $\Delta$ is the
  ideal $\mc{KZ}(\Delta)$ generated by the equivariant $K$-theoretic
  fundamental classes of the irreducible components of $Z(\Delta)$.
\end{defn}
If $V(x_{i_1}, \ldots , x_{i_k}) \subset \Aff^n$ is a linear subspace then  
$[V] = \prod_{j=1}^k (1- e_{i_j}^{-1})$
in $K_{\GG_m}(\Aff^n) = \ZZ[e_1,e_1^{-1}, \ldots , e_{n}, e_{n}^{-1}]$.
Using the excision sequence for equivariant $K$-theory, we obtain the analogue
of Theorem \ref{thm.chows-r}:
\begin{thm} \label{thm.kts-r}
With notation as above,
\[{K^0(\mc{X}(\Delta)) \cong R(G)/\mc{KZ}(\Delta) \cong Z[e_1,e_1^{-1},
  \ldots , e_n, e_n^{-1}]/(KL(\Delta) +\mc{KZ}(\Delta))}.\]
\end{thm}
\begin{exa}
  Once again, let $\sigma$ be the cone of Example \ref{ex.basiccone},
  and let $\sigma^*$ be its star subdivision as in Example \ref{exa.condivided}. Then $K(\X(\sigma^*))$ is the quotient of the character algebra of
  $\GG_m^5$,
  $\ZZ[e_1,e_1^{-1}, e_2,e_2^{-1}, e_3,e_3^{-1}, e_4,e_4^{-1}, e,e^{-1}]$,
by the $K$-theoretic ideal of linear relations and the $K$-theoretic Stanley-Reisner ideal. 
In our case, the $K$-theoretic 
ideal of linear relations is generated by
$(e_1e_3^{-1} -1, e_2e_4^{-1} -1, e_1e_2e_3e_4e-1)$, and the $K$-theoretic Stanley-Reisner ideal is generated
  by $(1 - (e_1e_3)^{-1}, 1 - (e_2e_4)^{-1})$.
  Thus, $$K_0(\X(\sigma^*)) \simeq \ZZ[e_1,e_2]/(e_2^2-1, e_1^2-1) =
  K_0(\PP^1) \otimes K_0(\PP^1).$$
 \end{exa}
\subsection{Completion of the proof}
The proof now proceeds as the proof of Theorem
\ref{thm.Chowcoh}. Using the analogue of Kimura's exact sequence for
operational $K$-theory \cite{AnPa:15} and Proposition
\ref{prop.kpullback}, we are reduced to showing (with the same notation
as in Section \ref{sec.excep}) that the restriction $K^0({\mathcal E})
\to K^0(\X)$ is an isomorphism. This fact follows from the
Stanley-Reisner description of $K(\X)$ and the same formal argument
used in the proof of Theorem \ref{thm.epullback}.

Finally, if $X$ is any smooth variety such that $A^*(X)$ is torsion
free, then $K(X)$ is also torsion free. This follows from the fact that
the natural map $A_k(X) \to \Gr_k K^0(X)$, $[V] \mapsto [{\mathcal O}_V]$ is surjective and becomes an isomorphism after tensoring
with $\QQ$ \cite[Examples 15.1.5, 15.2.16]{Ful:84}. Hence, the same
argument used in the proof of Theorem \ref{thm.chowtorfree} imples
that $\op K^0(X(\Delta))$ is torsion free for any semi-proper toric
variety. In particular, $\op K^0(X(\sigma))$ is torsion free for an
affine toric variety.

\begin{exa} \label{ex.final}
  Let $\sigma$ be the cone of Example \ref{ex.basiccone}. In Example
  \ref{exa.opnothom}, we showed that 
  $\rk A_*(X(\sigma))$ $\geq$ $2$. Thus, by the Riemann-Roch theorem for singular schemes
  $\rk K_0(X(\sigma)) \geq 2$, hence $\op K^0(X(\sigma)) \neq
  K_0(X(\sigma))$ even after tensoring with $\QQ$.
  \end{exa}
\bibliographystyle{amsmath}
\def\cprime{$'$}
\def\cprime{$'$}
\def\cprime{$'$} \def\cprime{$'$} \def\cprime{$'$}

\end{document}